\documentclass[12pt, reqno]{amsart}

\author[P.~Leonetti]{Paolo Leonetti}
\address{Universit\`a ``Luigi Bocconi''\\Department of Decision Sciences\\Milan, Italy}
\email{paolo.leonetti@unibocconi.it}

\author[F.~Maccheroni]{Fabio Maccheroni}
\address{Universit\`a ``Luigi Bocconi''\\Department of Decision Sciences\\Milan, Italy}
\email{fabio.maccheroni@unibocconi.it}

\title{Characterizations of Ideal Cluster Points}

\usepackage{amsmath}
\usepackage{amssymb}
\usepackage{amsthm}
\usepackage[left=3.5cm, right=3.5cm, paperheight=11.8in]{geometry}
\usepackage{hyperref}
\usepackage{url}
\usepackage{fancyhdr}
\usepackage{enumitem}
\usepackage{comment}
\usepackage{mathrsfs}
\usepackage{graphicx}
\usepackage[utf8]{inputenc}

\AtBeginDocument{%
   \def\MR#1{}
}

\newtheorem{thm}{Theorem}[section]
\newtheorem{cor}[thm]{Corollary}%[section]
\newtheorem{lem}[thm]{Lemma}

\theoremstyle{definition} 
\newtheorem{defi}[thm]{Definition}%[section]
\let\olddefi\defi
\renewcommand{\defi}{\olddefi\normalfont}
\newtheorem{example}[thm]{Example}
\let\oldexample\example
\renewcommand{\example}{\oldexample\normalfont}
\hypersetup{
    pdftitle={Characterizations of Ideal Cluster Points},
    pdfauthor={Paolo Leonetti and Fabio Maccheroni},
    pdfmenubar=false,
    pdffitwindow=true,
    pdfstartview=FitH,
    colorlinks=true,
    linkcolor=blue,
    citecolor=green,
    urlcolor=cyan
}

\providecommand{\MR}[1]{}
\providecommand{\bysame}{\leavevmode\hbox to3em{\hrulefill}\thinspace}
\providecommand{\MR}{\relax\ifhmode\unskip\space\fi MR }
\providecommand{\href}[2]{#2}
                                 
\begin{document}

\maketitle
\thispagestyle{empty}

\begin{abstract}
Given an ideal $\mathcal{I}$ on $\omega$, we prove that a sequence in a topological space $X$ is $\mathcal{I}$-convergent if and only if there exists a ``big'' $\mathcal{I}$-convergent subsequence. Then, we study several properties and show two characterizations of the set of $\mathcal{I}$-cluster points as classical cluster points of a filters on $X$ and as the smallest closed set containing ``almost all'' the sequence. As a consequence, we obtain that the underlying topology $\tau$ coincides with the topology generated by the pair $(\tau,\mathcal{I})$. 
\end{abstract}
%%%%%%%%%%%%%%%%%%%%%%%%%%%%%%%

\section{Introduction}\label{sec:introduction}

Following the concept of statistical convergence as a generalization of the ordinary convergence, Fridy \cite{MR1181163} introduced the statistical limit points and statistical cluster points of a real sequence $(x_n)$ as generalizations of accumulation points.

A real number $\ell$ is said to be a \emph{statistical limit point} of $(x_n)$ if there exists a subsequence $(x_{n_k})$ such that 
$$\lim_{k\to \infty} x_{n_k} = \ell$$ and the set of indices $\{n_k: k \in \omega\}$ has positive upper asymptotic density (see Section \ref{sec:preliminaries} for definitions). Also, $\ell$ is called \emph{statistical cluster point} provided that $$\{n\in \omega: |x_n-\ell|<\varepsilon\}$$ has positive upper asymptotic density for every $\varepsilon>0$. He proved, among others, that these concepts are not equivalent.

These notions have been studied in a number of recent papers, see e.g. \cite{PaoloMarek17, MR1372186, MR1416085, Leo17, LMM18, MR1260176, MR1758553}. Extensions of statistical convergence to more general spaces can be found in \cite{MR2904078, MR2463821, MR938459, MR1821765}, and to ideal convergence, see e.g. \cite{MR2835960, MR2320288, MR1838788, Leoproj}.

Given an ideal $\mathcal{I}$ on the positive integers $\omega$, we 
investigate various properties of $\mathcal{I}$-cluster points and $\mathcal{I}$-limit points of sequences taking values in topological spaces $(X,\tau)$. The main contributions of the article are: 
\begin{enumerate}[label=(\roman*)]
\item \label{itema1} a new characterization of $\mathcal{I}$-convergence: informally, a sequence $(x_n)$ is $\mathcal{I}$-convergent if and only if there exists a ``big'' $\mathcal{I}$-convergent subsequence (see Theorem \ref{lem:basic0}.\ref{item:basic4} and Corollary \ref{cor:cor1});
\item \label{itema2} the topology generated by the pair $(\tau,\mathcal{I})$ corresponds to the underlying topology $\tau$ (see Theorem \ref{thm:sametopology});
\item \label{itema3} a characterization of $\mathcal{I}$-cluster points as classical ``cluster points of the filter'' generated by the sequence (see Theorem \ref{thm:characterizationbourbaki});
\item \label{itema4} a characterization of the set of $\mathcal{I}$-cluster points as the smallest closed set containing ``almost all'' the sequence (see Theorem \ref{thm:char2}).
\end{enumerate}

%%%%%%%%%%%%%%%%%%%%%%%%%%%%%%%%%%%%%%%%%%%%%%%%%%%%%%%%%%%%%%%%%%%%%%%%%%%
\section{Preliminaries}\label{sec:preliminaries}

Let $\mathrm{Fin}$ be the collection of finite subsets of $\omega$. The upper asymptotic density of a set $S\subseteq \omega$ is defined by
$$
\mathrm{d}^\star(S):=\limsup_{n\to \infty} \frac{|S\cap [1,n]|}{n}\,
$$
and we denote by $\mathcal{Z}$ the collection of all $S$ such that $\mathrm{d}^\star(S)=0$. Hence, a real number $\ell$ is a statistical cluster point of a given real sequence $(x_n)$ if and only if $\{n\in \omega: |x_n-\ell|<\varepsilon\}$ does not belong to $\mathcal{Z}$ for every $\varepsilon>0$.

An ideal $\mathcal{I}$ on $\omega$ is a family of subsets of positive integers closed under taking finite unions and subsets of its elements. It is also assumed that $\mathcal{I}$ is different from the power set of $\omega$ and contains all the singletons. It is clear that $\mathrm{Fin}$ and $\mathcal{Z}$ are ideals. Many other examples can be found, e.g., in \cite[Chapter 1]{MR1711328} and \cite[Section 2]{MR3594409}. Intuitively, an ideal represents the collection of subsets of $\omega$ which are ``small'' in a suitable sense. %In addition, w
We denote by $\mathcal{I}^\star:=\{A\subseteq \omega: A^c \in \mathcal{I}\}$ the \emph{filter dual} of $\mathcal{I}$ and by $\mathcal{I}^+$ the collection of $\mathcal{I}$\emph{-positive sets}, that is, the collection of all sets which do not belong to $\mathcal{I}$. 

\begin{defi}\label{iconverg}
Given a topological space $X$, a sequence 
$x=(x_n)$ is said to be $\mathcal{I}$\emph{-convergent} to $\ell$, shortened with $x_n \to_{\mathcal{I}} \ell$, whenever $\{n: x_n \in U\} \in \mathcal{I}^\star$ for all neighborhoods $U$ of $\ell$. 
Moreover, let $\Gamma_x(\mathcal{I})$ denote the set of $\mathcal{I}$\emph{-cluster points} of $x$, that is, the set of all $\ell \in X$ such that $\{n: x_n \in U\} \in \mathcal{I}^+$ for all neighborhoods $U$ of $\ell$. 
\end{defi}

Ordinary convergence corresponds to $\mathrm{Fin}$-convergence (thus, we shorten $x_n \to_{\mathrm{Fin}} \ell$ with $x_n \to \ell$) and statistical convergence to $\mathcal{Z}$-convergence. Now, one may worder whether $\mathcal{I}$-convergence corresponds to ordinary convergence with respect to another topology on the same base set. Essentially, it never happens. %this is never the case.
\begin{example}
Let us assume that $\mathcal{I}\neq \mathrm{Fin}$ and $X$ is a topological space with at least two distinct points such that its topology $\tau$ is not the trivial topology $\tau_0$. 
Hence, there exists a set $I\in \mathcal{I}\setminus \mathrm{Fin}$; in particular, $I$ is infinite. Fix distinct $a,b \in X$ and define the sequence $(x_n)$ by 
$x_n=a$ whenever $n \notin I$ and $x_n=b$ otherwise. 
It follows by construction that $x_n \to_\mathcal{I} a$ in $(X,\tau)$. Let us assume, for the sake of contradiction, there exists a topology $\tau^\prime$ such that $x_n \to a$ in $(X,\tau^\prime)$. If there is a $\tau^\prime$-neighborhood $U$ of $a$ such that $b\notin U$, then 
$
\{n: x_n\notin U\} =I. 
$ 
This is impossible, since $I$ is not finite. Hence $b \in U$ whenever $a \in U$. By the arbitrariness of $a$ and $b$, we conclude that $\tau^\prime=\tau_0$. The converse is false: given $U \in \tau\setminus \tau_0$ and $u \in U$, then the constant sequence $(u)$ is not $\mathcal{I}$-convergent to $\ell$ provided that $\ell\notin U$.
\end{example}

Other notions of convergence have been defined in literature, considering properties of subsequences of $x$ with sufficiently many elements. 
\begin{defi}\label{def:istarconverg}
Given a topological space $X$, a sequence 
$x=(x_n)$ is said to be $\mathcal{I}^\star$\emph{-convergent} to $\ell$, shortened with $x_n \to_{\mathcal{I}^\star} \ell$, whenever there exists a subsequence $(x_{n_k})$ such that $x_{n_k} \to \ell$ and $\{n_k: k \in \omega\} \in \mathcal{I}^\star$. 
Moreover, let $\Lambda_x(\mathcal{I})$ denote the set of $\mathcal{I}$\emph{-limit points} of $x$, that is, the set of all $\ell \in X$ such that there exists a subsequence $(x_{n_k})$ for which $x_{n_k} \to \ell$ and $\{n_k: k \in \omega\} \in \mathcal{I}^+$. 
\end{defi}

At this point, recall that an ideal $\mathcal{I}$ is a \emph{P-ideal} if it is $\sigma$-directed modulo finite sets, i.e., for every sequence $(A_n)$ of sets in $\mathcal{I}$ there exists $A \in \mathcal{I}$ such that $A_n\setminus A$ is finite for all $n$; equivalent definitions were given, e.g., in \cite[Proposition 1]{MR2285579}. 

Moreover, given infinite sets $A,B \subseteq \omega$ such that $A$ has canonical enumeration $\{a_n: n \in \omega\}$, we say that $\mathcal{I}$ a \emph{G-ideal} if 
$$
A_B:=\{a_b: b \in B\} \in \mathcal{I}^\star\,\,\,\text{ if and only if }\,\,\,B\in \mathcal{I}^\star
$$
provided that $A \in \mathcal{I}^\star$. This condition is strictly related to the so-called ``property (G)'' considered in \cite{MR3568092} and to the definition of invariant and thinnable ideals considered in \cite{Leo17, Leo17b}. Note that the class of G-ideals contains the ideals generated by $\alpha$-densities with $\alpha \ge -1$ (in particular, $\mathcal{I}_\mathrm{d}$ and the collection of logarithmic density zero sets), several summable ideals, and the \emph{P\'olya ideal}, i.e., 
$$
\mathcal{I}_{\mathfrak{p}}:=\left\{S\subseteq \omega: \mathfrak p^\star(S):=\lim_{s \to 1^-} \limsup_{n \to \infty} \frac{|S \cap [ns,n]|}{(1-s)n}=0\right\},
$$
see \cite[Section 2]{Leo17}. Among other things, the upper P\'olya density $\mathfrak p^\star$ has found a number of remarkable applications in analysis and economic theory, see e.g. \cite{MR1545027}, \cite{MR0003208} and \cite{MR1656470}.

In this regard, we have the following basic result: points \ref{item:basic1}-\ref{item:basic2} can be shown by routine arguments, cf. \cite[Theorem 3.1]{MR2904078} and \cite[Section 2]{MR2463821} (we omit details); although not explicit in the literature, point \ref{item:basic3} can be considered folklore, see \cite[Theorem 3.2]{MR1844385} for the case $X$ being a metric space (we include the proof here for the sake of completeness); lastly, point \ref{item:basic4} provides a new characterization of $\mathcal{I}$-convergence (related results can be found in \cite[Theorem 3.4]{MR3568092} and \cite[Theorem 3.4]{Leo17}).

\begin{thm}\label{lem:basic0}
Let $X$ be a topological space and $\mathcal{I}$ be an ideal. Then:
\begin{enumerate}[label={\rm (\roman{*})}]
\item \label{item:basic1} $\mathcal{I}$-limits and $\mathcal{I}^\star$-limits are unique, provided $X$ is Hausdorff;
\item \label{item:basic2} $\mathcal{I}^\star$-convergence implies $\mathcal{I}$-convergence;
\item \label{item:basic3} $\mathcal{I}$-convergence implies $\mathcal{I}^\star$-convergence, provided $X$ is first countable and $\mathcal{I}$ is a P-ideal;
\item \label{item:basic4} A sequence $(x_n) \in X^{\omega}$ is $\mathcal{I}$-convergent if and only if there exists an $\mathcal{I}$-convergent subsequence $(x_{n_k})$ such that $\{n_k: k \in \omega\} \in \mathcal{I}^\star$, provided $\mathcal{I}$ is a G-ideal.
\end{enumerate}
\end{thm}
\begin{proof}
\ref{item:basic3} Let $(x_n)$ be a sequence taking values in $X$ which is $\mathcal{I}$-convergent to some $\ell \in X$. Then, let $(U_j)$ be a countable decreasing local base at $\ell$ and, for each $j$, define $A_j:=\{n: x_n\notin U_j\}$. Hence, $A_j \in \mathcal{I}$ for each $j$, $(A_j)$ is increasing, and, since $\mathcal{I}$ is a P-ideal, there exists $A \in \mathcal{I}$ such that $A_j\setminus A$ is finite for all $j$. Denoting by $(n_k)$ the increasing sequence of integers in $A^c$ (which belongs to $\mathcal{I}^\star$), it follows that $x_{n_k} \to \ell$. Indeed, letting $V$ be a neighborhood of $\ell$ and $j \in \omega$ such that $U_j\subseteq V$, then the finiteness of $\{k: x_{n_k}\notin V\}$ follows by the fact that it has the same cardinality of $\{n_k: x_{n_k}\notin V\}$ and
$
\{n_k: x_{n_k}\notin V\} \subseteq \{n_k: x_{n_k}\notin U_j\} \subseteq \{n \in A^c: x_n \notin U_j\} = A_j\setminus A. 
$

\ref{item:basic4} Let us suppose that $(x_n)$ is $\mathcal{I}$-convergent to $\ell \in X$. Fix also $I \in \mathcal{I}$ and let $(n_k)$ be the increasing enumeration of $I^c$. Then, it is claimed that the subsequence $(x_{n_k})$ is $\mathcal{I}$-convergent to $\ell$. Indeed, for each neighborhood $U$ of $\ell$, we have $\{n: x_n \notin U\} \in \mathcal{I}$ by hypothesis, hence 
$
\{n_k: x_{n_k} \in U\} = \{n: x_n \in U\}\setminus I=\omega\setminus (\{n: x_n \notin U\} \cup I) \in \mathcal{I}^\star. 
$ 
It follows by the fact that $\mathcal{I}$ is a G-ideal that $\{k: x_{n_k} \in U\} \in \mathcal{I}^\star$, that is, $x_{n_k} \to_{\mathcal{I}} \ell$. 
The converse can be shown similarly.
\end{proof}

It is well known that $\mathcal{Z}$ is a P-ideal (see e.g. \cite[Proposition 3.2]{MR632187}) and, as recalled before, it is also a G-ideal. Hence:%Accordingly, we easily obtain:
\begin{cor}\label{cor:cor1}
Let $(x_n)$ be a sequence taking values in a topological space $X$. Then the following are equivalent:
\begin{enumerate}[label={\rm (\roman{*})}]
\item \label{item:eq1} $(x_n)$ is statistically convergent;
\item \label{item:eq3} There exists a statistically convergent subsequence $(x_{n_k})$ with $\{n_k: k \in \omega\} \in \mathcal{Z}^\star$.
\end{enumerate}
If, in addition, $X$ is first countable, then they are also equivalent to:
\begin{enumerate}[label={\rm (\roman{*})}]
\setcounter{enumi}{2}
\item \label{item:eq2} There exists a convergent subsequence $(x_{n_k})$ with $\{n_k: k \in \omega\} \in \mathcal{Z}^\star$;
\end{enumerate}
\end{cor}
It is worth noting that the equivalence between \ref{item:eq1} and \ref{item:eq2} can be already found in \cite[Theorem 2.2]{MR2463821}, cf. also \cite[Theorem 1]{MR816582} and \cite[Theorem 1]{MR1260176}. 

We obtain also an abstract version of \cite[Theorem 2.3]{MR954458}, see also \cite[Proposition 1]{MR2835960} and \cite[Theorem 1]{MR2334006}; the proof goes verbatim, hence we omit it.
\begin{cor}\label{cor:decomposition}
Let $\mathcal{I}$ be a P-ideal and $(x_n)$ be a sequence taking values in a 
%first countable topological 
metrizable %Szymon
group \textup{(}with identity $0$\textup{)} such that $x_n \to_\mathcal{I} \ell$. Then, there exist sequences $(y_n)$ and $(z_n)$ such that: $x_n=y_n+z_n$ for all $n$, $y_n \to \ell$, and $\{n\in \omega: z_n\neq 0\} \in \mathcal{I}$. % (in particular, $z_n \to_{\mathcal{I}^\star} 0$).
\end{cor}
%It is worth noting that Corollary \ref{cor:decomposition} is also an extension of \cite[Theorem 2]{MR2002719}.

Recall that a real double sequence $x=(x_{n,m}: n,m \in\omega)$ has \emph{Pringsheim limit} $\ell$ provided that for every $\varepsilon>0$ there exists $k \in \omega$ such that $|x_{n,m}-\ell|<\varepsilon$ for all $n,m\ge k$. Identifying ideals on countable sets with ideals on $\omega$ through a fixed bijection, it is easily seen that this is equivalent to $x \to_{\mathcal{I}_{\mathrm{Pr}}} \ell$, where $\mathcal{I}_{\mathrm{Pr}}$ is the ideal defined by
$$
\mathcal{I}_{\mathrm{Pr}}:=\left\{A\subseteq \omega \times \omega: \limsup_{n\to \infty}\, \sup \left\{k: (n,k) \in A\right\}<\infty\right\}.
$$
Equivalently, $\mathcal{I}_{\mathrm{Pr}}$ is the ideal on $\omega \times \omega$ containing the complements of $[n,\infty)\times [n,\infty)$ for all $n \in \omega$. At this point, for each $n,m \in \omega$, let $\mu_{n,m}$ be the uniform probability measure on $\{1,\ldots,n\}\times \{1,\ldots,m\}$ and define the ideal
$$
\mathcal{Z}_{\mathrm{Pr}}:=\left\{A\subseteq \omega\times \omega: \mu_{n,m}(A) \to_{\mathcal{I}_{\mathrm{Pr}}} 0\right\}.
$$
Note that $\mathcal{I}_{\mathrm{Pr}}\subseteq \mathcal{Z}_{\mathrm{Pr}}$ and that $\mathcal{Z}_{\mathrm{Pr}}$ is a P-ideal. The notion of convergence of real double sequences $(x_{n,m})$ with respect to the ideal $\mathcal{Z}_{\mathrm{Pr}}$ has been recently introduced in \cite{MR2002719, MR2019757}; here, it has been simply defined ``statistical convergence'' of double sequences. Accordingly, it has been shown in \cite[Theorem 2]{MR2002719} that a real double sequence $(x_{n,m})$ is statistically convergent to $\ell$ if and only if there exist real double sequences $(y_{n,m})$ and $(z_{n,m})$ such that $y_{n,m}\to_{\mathcal{I}_{\mathrm{Pr}}} \ell$ and $\{(n,m): z_{n,m}\neq 0\}\in \mathcal{Z}_{\mathrm{Pr}}$. However, this is an immediate consequence of Corollary \ref{cor:decomposition}.

%%%%%%%%%%%%%%%%%%%%%%%%%%%%%%%%%%%%%%%%%%%%%%%%%%%%%%%%%%%%%%%%%%%%%%%%%%%%%%%%%%%%%%%%%%%%%%%%%%%
\section{Ideal Cluster points}\label{sec:cluster}

Given sequences $x$ and $y$ taking values in a topological space $X$, we say that they are $\mathcal{I}$\emph{-equivalent}, shortened with $x\equiv_\mathcal{I} y$, if $\{n: x_n \neq y_n\} \in \mathcal{I}$ (it is easy to see that $\equiv_\mathcal{I}$ is an equivalence relation). 
The following lemmas, which collect and extend several results contained in \cite{MR2463821, MR1181163, MR1844385}, show some standard properties of $\mathcal{I}$-cluster and $\mathcal{I}$-limit points.
\begin{lem}\label{lem:basic}
Let $x$ and $y$ be sequences taking values in a topological space $X$ and fix ideals $\mathcal{I}\subseteq \mathcal{J}$. Then:
\begin{enumerate}[label={\rm (\roman{*})}]%\textsc{l}\arabic
\item \label{item:1} $\Lambda_x(\mathcal{J}) \subseteq \Lambda_x(\mathcal{I})$ and $\Gamma_x(\mathcal{J}) \subseteq \Gamma_x(\mathcal{I})$;
\item \label{item:2} $\Lambda_x(\mathrm{Fin}) = \Gamma_x(\mathrm{Fin})$, provided $X$ is first countable;
\item \label{item:3} $\Lambda_x(\mathcal{I}) \subseteq \Gamma_x(\mathcal{I})$;
\item \label{item:4} $\Gamma_x(\mathcal{I})$ is closed;
\item \label{item:5} $\Lambda_x(\mathcal{I})=\Lambda_y(\mathcal{I})$ and $\Gamma_x(\mathcal{I})=\Gamma_y({\mathcal{I}})$ provided $x\equiv_{\mathcal{I}} y$;
\item \label{item:6} $\Gamma_x(\mathcal{I}) \cap K \neq \emptyset$, provided $K\subseteq X$ is compact and $\{n: x_n \in K\} \in \mathcal{I}^+$;
\item \label{item:7} $\Lambda_x(\mathcal{I}) = \Gamma_x(\mathcal{I})=\{\ell\}$ provided $x_n \to_{\mathcal{I}^\star} \ell$ and $X$ is Hausdorff.
\end{enumerate} 
\end{lem}
\begin{proof}
\ref{item:1} and \ref{item:2} easily follow from the definitions. 
In addition, \ref{item:3} is obvious if $\Lambda_x(\mathcal{I})=\emptyset$. Otherwise, fix $\ell \in \Lambda_x(\mathcal{I})$ and a neighborhood $U$ of $\ell$. Then, there exists an increasing subsequence $(n_k)$ with $\{n_k\}\in\mathcal{I}^+$ such that $x_{n_k} \to \ell$, so that $S:=\{n_k: x_{n_k} \notin U\}$ is finite. This implies that $\{n_k\} \setminus S\subseteq \{n: x_n \in U\}$. To conclude, it is sufficient to note that $\{n_k\}\setminus S \notin \mathcal{I}$, therefore $\{n: x_n \in U\}\in \mathcal{I}^+$.

Similarly, \ref{item:4} is clear if $\Gamma_x(\mathcal{I})=\emptyset$. In the opposite, let $y$ be an accumulation point of $\Gamma_x(\mathcal{I})$ and $U$ a neighborhood of $y$. Then, there exists $z \in \Gamma_x(\mathcal{I}) \cap U$. Let $V$ be a neighborhood of $z$ contained in $U$. Considering that $\{n: x_n \in V\} \subseteq \{n: x_n \in U\}$ and $\{n: x_n\in V\} \in \mathcal{I}^+$, we conclude that $y \in \Gamma_x(\mathcal{I})$.

To prove \ref{item:5}, fix $\ell \in \Lambda_x(\mathcal{I})$, so that there exists a subsequence $(x_{n_k})$ such that $\{n_k\}\in \mathcal{I}^+$ and $x_{n_k} \to \ell$. Since $\{n: x_n\neq y_n\} \in \mathcal{I}$ and $\{n_k: x_{n_k} \neq y_{n_k}\} \subseteq \{n: x_n\neq y_n\}$, then $S:=\{n_k: x_{n_k} = y_{n_k}\} \in \mathcal{I}^+$. Denoting by $(s_n)$ the canonical enumeration of $S$, we obtain $y_{s_n} \to \ell$, hence $\ell \in \Lambda_y({\mathcal{I}})$. By the arbitrariness of $\ell$, we have $\Lambda_x(\mathcal{I})\subseteq \Lambda_y({\mathcal{I}})$ therefore, by symmetry, $\Lambda_x(\mathcal{I})=\Lambda_y({\mathcal{I}})$. 
The other claim can be shown similarly.

The proof of \ref{item:6} can be found in \cite[Theorem 6]{MR2923430}, cf. also \cite[Theorem 2.14]{MR2463821} for the case $\mathcal{I}=\mathcal{Z}$.

Lastly, suppose that $x_n \to_{\mathcal{I}^\star} \ell$ so that $x_n \to_{\mathcal{I}} \ell$ by Theorem \ref{lem:basic0}.\ref{item:basic2} and, in particular, $\ell \in \Lambda_x(\mathcal{I})$. Also, thanks to \ref{item:3}, we have $\{\ell\}\subseteq \Lambda_x(\mathcal{I}) \subseteq \Gamma_x(\mathcal{I})$. To conclude, let us suppose for the sake of contradition that there exists an $\mathcal{I}$-cluster point $\ell^\prime$ of $x$ different from $\ell$. Fix disjoint neighborhoods $U$ and $U^\prime$ of $\ell$ and $\ell^\prime$, respectively. On the one hand, since $\ell^\prime$ is a $\mathcal{I}$-cluster point, then $\{n:x_n \in U^\prime\} \in \mathcal{I}^+$. On the other hand, this is impossible since $\{n:x_n \in U^\prime\} \subseteq \{n: x_n\notin U\} \in \mathcal{I}$. This proves \ref{item:7}.
\end{proof}

It follows at once from Theorem \ref{lem:basic0}.\ref{item:basic3} and Lemma \ref{lem:basic}.\ref{item:7} that:
\begin{cor}\label{lem:ilimit}
Let $\mathcal{I}$ be a P-ideal and $(x_n)$ be a sequence taking values in a first countable Hausdorff space such that $x_n \to_\mathcal{I} \ell$. Then $\Lambda_x(\mathcal{I}) = \Gamma_x(\mathcal{I})=\{\ell\}$.
\end{cor}

The converse of Corollary \ref{lem:ilimit} does not hold in general: the real sequence $x$ defined by $x_n=n$ if $n$ is even and $x_n=0$ otherwise satisfies $\Lambda_x(\mathcal{Z}) = \Gamma_x(\mathcal{Z})=\{0\}$ while $x_n \not\to_{\mathcal{Z}} 0$. On the other hand, if the underlying space space is compact, it is sufficient, cf. \cite[Proposition 8]{MR1757066} for a special case.
\begin{lem}\label{lem:converseabovelemmaconvergence}
Let $\mathcal{I}$ be an ideal, let $(x_n)$ be a sequence in a first countable compact space $X$, and suppose that $\Gamma_x(\mathcal{I})=\{\ell\}$. Then $x_n \to_{\mathcal{I}} \ell$. In addition, if $\mathcal{I}$ is a P-ideal, then $x_n \to_{\mathcal{I}^\star} \ell$.
\end{lem}
\begin{proof}
Let $(U_k)$ be a decreasing local base at $\ell$. Fix $k \in \omega$ and, for each $z \in X$ with $z\neq \ell$, there exists a neighborhood $U_z$ of $z$ such that $\{n \in \omega: x_n \in U_z\} \in \mathcal{I}$. Since $\{U_z: z \in X\setminus \{\ell\}\} \cup U_k$ is an open cover of $X$ and $X$ is compact, there exists a finite subcover $U_{z_1}\cup \cdots \cup U_{z_m} \cup U_k$; note that $U_k$ belongs to the subcover, indeed, in the opposite, we would have $\omega=\bigcup_{i\le m}\{n: x_n \in U_{z_i}\} \in \mathcal{I}$. In particular, $\{n \in \omega: x_n \in U_k\} \in \mathcal{I}^\star$. Therefore $x_n \to_{\mathcal{I}} \ell$.

If, in addition, $\mathcal{I}$ is a P-ideal then $A_k:=\{n\in \omega: x_n \notin U_k\}$ is an increasing sequence in $\mathcal{I}$, hence there exists $A \in \mathcal{I}$ such that $A_k\setminus A \in \mathrm{Fin}$ for all $k$. It follows that $\{n \in A^c: x_n \notin U_k\}=A_k \cap A^c \in \mathrm{Fin}$ for all $k$, that is, $x_n \to_{\mathcal{I}^\star} \ell$.
\end{proof}

As an application, we obtain a generalization of \cite[Theorem 3]{MR1416085}:
\begin{cor}
Let $\mathcal{I}$ be an ideal and $(x_n)$ be a sequence in first countable space $X$ such that $\{n \in \omega: x_n \notin K\} \in \mathcal{I}$ for some compact $K\subseteq X$. Then $x_n \to_{\mathcal{I}} \ell$ if and only if $\Gamma_x(\mathcal{I})=\{\ell\}$.
\end{cor}

Moreover, Lemma \ref{lem:basic}.\ref{item:5} can be strenghtened if $X$ is a topological group:
\begin{lem}\label{lem:strengtened}
Let $x$ and $y$ be sequences taking values in a topological group $X$ \textup{(}written additively, with identity $0$\textup{)} and fix an ideal $\mathcal{I}$. Then:
\begin{enumerate}[label={\rm (\roman{*})}]
\item \label{item:lem1} $\Gamma_x(\mathcal{I})=\Gamma_y(\mathcal{I})$ provided $x_n-y_n \to_{\mathcal{I}} 0$; 
\item \label{item:lem2} $\Lambda_x(\mathcal{I})=\Lambda_y(\mathcal{I})$ provided $x_n-y_n \to_{\mathcal{I}^\star} 0$. 
\end{enumerate}
\end{lem}
\begin{proof} Let $z$ be the sequence defined by $z_n=x_n-y_n$.

\ref{item:lem1} It follows by hypothesis $z_n \to_{\mathcal{I}} 0$ and $-z_n \to_{\mathcal{I}} 0$. Fix $\ell \in \Gamma_x(\mathcal{I})$ and let $U$ be a neighborhood of $\ell$. By the continuity of the operation of the group, there exist neighborhoods $V$ and $W$ of $\ell$ and $0$, respectively, such that $V+W \subseteq U$. Considering that $\{n: x_n \in V\} \in \mathcal{I}^+$ and $\{n: -z_n \in W\} \in \mathcal{I}^\star$, it follows that
$$
\{n: y_n \in U\}=\{n: x_n-z_n \in U\} \supseteq \{n: x_n \in V\} \cap \{n: -z_n \in W\} \in \mathcal{I}^+.
$$
Since $\ell$ and $U$ were arbitrarily chosen, then $\Gamma_x(\mathcal{I})\subseteq \Gamma_y(\mathcal{I})$. The opposite inclusion can be shown similarly. 

\ref{item:lem2} By hypothesis $z_n \to_{\mathcal{I}^\star} 0$ and $-z_n \to_{\mathcal{I}^\star} 0$. Fix $\ell \in \Lambda_x(\mathcal{I})$, hence there exist $A,B \in \mathcal{I}^\star$ such that $\lim_{a \in A}x_a=\ell$ and $\lim_{b \in B}-z_b=0$. Setting $C:=A\cap B \in \mathcal{I}^\star$, it follows that $\lim_{c \in C}y_c=\lim_{c \in C}x_c-z_c=\ell$, therefore $\Lambda_x(\mathcal{I})\subseteq \Lambda_y(\mathcal{I})$. The opposite inclusion can be shown similarly.
\end{proof}

We recall that, under suitable assumptions on $X$ and $\mathcal{I}$, the collection of $\mathcal{I}$-cluster and $\mathcal{I}$-limit point sets can be characterized as the closed sets and $F_\sigma$ sets, respectively; see \cite[Theorem 3.1]{PaoloMarek17}, \cite[Section 2]{MR2463821}, \cite[Theorem 1.1]{MR1838788}, and \cite[Section 4]{MR1844385}. Moreover, the continuity of the map $x\mapsto \Gamma_x(\mathcal{I})$ has been investigated in \cite{MR1838788}.%, MR2177419

The next result establishes a connection between sets of cluster points with respect to different ideals (the proof is based on \cite[Theorem 2]{MR1181163} which focuses on the case $X=\mathbf{R}$, $\mathcal{I}=\mathcal{Z}$, and $\mathcal{J}=\mathrm{Fin}$).
\begin{lem}\label{thm:Clrelation}
Let $x$ be a sequence taking values in a strongly Lindel\"{o}f space $X$ and fix ideals $\mathcal{J}\subseteq \mathcal{I}$ such that $\mathcal{I}$ is a P-ideal. Then, there exists an $\mathcal{I}$-equivalent sequence $y$ such that $\Gamma_x(\mathcal{I})=\Gamma_y(\mathcal{J})$ and $\{y_n: n \in \omega\}\subseteq \{x_n: n \in \omega\}$.
\end{lem}
\begin{proof}
The claim is obvious if $\Gamma_x(\mathcal{I})=\Gamma_x({\mathcal{J}})$. Hence, let us suppose that $\Delta:=\Gamma_x({\mathcal{J}})\setminus \Gamma_x(\mathcal{I})\neq \emptyset$ and, for each $z \in \Delta$, let $U_z$ be a neighborhood of $z$ such that $\{n: x_n \in U_z\} \in \mathcal{I}$. Then $\bigcup U_z$ is an open cover of $\Delta$. Since $X$ is 
strongly Lindel\"{o}f, there exists a countable subset $\{z_k: k \in \omega\}\subseteq \Delta$ such that $\bigcup U_{z_k}$ is an open subcover of $\Delta$. Moreover, since $\mathcal{I}$ is a P-ideal, there exists $I \in \mathcal{I}$ such that $\{n: x_n \in U_{z_k}\}\setminus I$ is finite for all $k$. 
At this point, let $(i_n)$ be the canonical enumeration of $\omega\setminus I$ and define the sequence $y$ by $y_n=x_{i_n}$ if $n \in I$ and $y_n=x_n$ otherwise. Since $\{n: x_n\neq y_n\}\subseteq I \in \mathcal{I}$, then $x \equiv_\mathcal{I} y$, hence we obtain by Lemma \ref{lem:basic}.\ref{item:5} that $\Gamma_x(\mathcal{I})=\Gamma_y({\mathcal{I}})$. The claim follows by the fact that every $\mathcal{J}$-cluster point of $y$ is also an $\mathcal{I}$-cluster point, therefore $\Gamma_y(\mathcal{I})=\Gamma_y(\mathcal{J})$.
\end{proof}

Lastly, given a topological space $(X,\tau)$ and an ideal $\mathcal{I}$, define the family
$$
\textstyle \tau(\mathcal{I}):=\left\{F^c \subseteq X: F=\bigcup_{x \in F^{\omega}}\Gamma_x(\mathcal{I})\right\},
%\{G\subseteq X: x_n \to \ell \in G \implies \{n: x_n \notin G\} \in \mathcal{I}\}.
$$
that is, $F$ is $\tau(\mathcal{I})$\emph{-closed} if and only if it is the union of $\mathcal{I}$-cluster points of $F$-valued sequences. In particular, it is immediate that $\tau=\tau(\mathrm{Fin})$.
\begin{lem}\label{lem:easy}
$\tau \subseteq \tau(\mathcal{I})$.
\end{lem}
\begin{proof}
Let $F$ be a $\tau$-closed set. Thanks to Lemma \ref{lem:basic}.\ref{item:1}, we have
$$%\begin{equation}\label{eq:chain}
\textstyle F\subseteq \bigcup_{x \in F^{\omega}} \Gamma_x(\mathcal{I}) \subseteq \bigcup_{x \in F^{\omega}} \Gamma_x(\mathrm{Fin})=F,
$$%\end{equation}
where the first inclusion is obtained by choosing the constant sequence $(f)$, for each fixed $f \in F$. Therefore, $F^c\in \tau(\mathcal{I})$.\end{proof}
The converse holds under some additional assumptions:
\begin{thm}\label{thm:sametopology}
Assume that one of the following conditions holds:
\begin{enumerate}[label={\rm (\roman{*})}]
\item \label{item:top1} $X$ is sequentially strongly Lindel\"{o}f and $\mathcal{I}$ is a P-ideal;
\item \label{item:top2} $X$ is first countable.
\end{enumerate}
Then $\tau=\tau(\mathcal{I})$.
\end{thm}
\begin{proof}
Thanks to Lemma \ref{lem:easy}, it is sufficient to show that $\tau(\mathcal{I}) \subseteq \tau$. Let $F$ be a $\tau(\mathcal{I})$-closed set. Then, it is enough to show that if $\ell \in F$ is an $\mathcal{I}$-cluster point of some $F$-valued sequence $x$, it is also an ordinary limit point of some $F$-valued sequence $y$.

\ref{item:top1}  This follows directly by Lemma \ref{thm:Clrelation}, setting $\mathcal{J}=\mathrm{Fin}$.

\ref{item:top2} Let $(U_k)$ be a decreasing local base at $\ell$. Then, there exists a subsequence $(x_{n_k})$ converging to $\ell$: to this aim, set $S_k:=\{n: x_n\in U_k\}$ for each $k$, fix $n_1 \in S_1$ arbitrarily and, for each $k\in \omega$, define $n_{k+1}:=\min S_{k+1}\setminus \{1,\ldots,n_k\}$ (note that this is possible since each $S_k$ is infinite).
\end{proof}

%%%%%%%%%%%%%%%%%%%%%%%%%%%%%%%%%%%%%%%%%%%%%%%%%%%%%%%%%%%%%%%%%%%%%%%
\section{Characterizations}\label{sec:charact}

Given an ideal $\mathcal{I}$ and a sequence $x$ taking values in a topological space $X$, we define the $\mathcal{I}$\emph{-filter generated by} $x$ as
$$
\mathscr{F}_x(\mathcal{I}):=\left\{Y\subseteq X: \{n: x_n \notin Y\} \in \mathcal{I}\right\}.
$$
It is immediate that $\mathscr{F}_x(\mathcal{I})$ is a filter on $X$ with filter base 
$$
\mathcal{B}_x(\mathcal{I}):=\{\{x_n: n\notin I\}: I \in \mathcal{I}\}.
$$
In addition, if $\mathcal{I}=\mathrm{Fin}$, then $\mathscr{F}_x(\mathcal{I})$ coincides with the standard filter generated by $x$, cf. \cite[Definition 7, p.64]{MR1726779}. 

With this notation, we are going to show that $\ell$ is an $\mathcal{I}$-cluster point of $x$ if and only if it is a cluster point of the filter $\mathscr{F}_x(\mathcal{I})$, that is, $\ell$ lies in the closure of all sets in the filter base $\mathcal{B}_x(\mathcal{I})$, c.f. \cite[Definition 2, p.69]{MR1726779}.

\begin{lem}\label{lem:firstinclusion}
$\bigcap_{B \in \mathcal{B}_x(\mathcal{I})} \overline{B} \subseteq \Gamma_x(\mathcal{I})$.
\end{lem}
\begin{proof}
Let us suppose that $\ell \in \bigcap_{I \in \mathcal{I}}\overline{\{x_n: n\notin I\}}$, that is, for each $I \in \mathcal{I}$ there exists a subsequence $(x_{n_k})$ converging to $\ell$ such that $\{n_k: k \in \omega\} \cap I = \emptyset$. Suppose for the sake of contradiction that $\ell$ is not an $\mathcal{I}$-cluster point, i.e., there exists an open neighborhood $U$ of $\ell$ such that $J:=\{n: x_n \in U\}$ belongs to $\mathcal{I}$. Then, it follows that $\{x_n: n \notin J\} \in \mathcal{B}_x(\mathcal{I})$ hence
$$
\textstyle \ell \in \bigcap_{B \in \mathcal{B}_x(\mathcal{I})} \overline{B}\subseteq \overline{\{x_n: n \notin J\}} = \overline{\{x_n: x_n \notin U\}} \subseteq U^c,
$$
which is impossible since $\ell \in U$.
\end{proof}

However, if $X$ is first countable, then also the converse holds.
\begin{thm}\label{thm:characterizationbourbaki}
Let $\mathcal{I}$ be an ideal and $x$ be a sequence taking values in a first countable space $X$. Then $\Gamma_x(\mathcal{I})=\bigcap_{B \in \mathcal{B}_x(\mathcal{I})} \overline{B}$.
\end{thm}
\begin{proof}
Thanks to Lemma \ref{lem:firstinclusion}, it is sufficient to show that $\Gamma_x(\mathcal{I})\subseteq \bigcap_{B \in \mathcal{B}_x(\mathcal{I})} \overline{B}$. 
Let us suppose that $\ell$ is an $\mathcal{I}$-cluster point of $x$ and fix a decreasing local base $(U_k)$ at $\ell$, so that $S_k:=\{n: x_n \in U_k\} \in \mathcal{I}^+$ for all $k$. Fix also $I \in \mathcal{I}$ and note that $T_k:=S_k\setminus I \in \mathcal{I}^+$ for all $k$ (in particular, each $T_k$ is infinite). Then, we have to prove that $\ell \in \overline{\{x_n: n\notin I\}}$, i.e., there exists a subsequence $(x_{n_k})$ converging to $\ell$ such that $n_k \notin I$ for all $k$. To this aim, it is enough to fix $n_1 \in T_1$ arbitrarily and $n_{k+1}:=\min T_{k+1}\setminus \{1,\ldots,n_k\}$ for all $k \in \omega$. It follows by construction that $\lim_{k\to \infty}x_{n_k}=\ell$ and $n_k \notin I$ for all $k$.
\end{proof}

As a corollary, we obtain another proof of Lemma \ref{lem:basic}.\ref{item:4}, provided $X$ is first countable.

We conclude with another characterization of the set of $\mathcal{I}$-cluster points, which subsumes the results contained in \cite{MR2040222}.
\begin{thm}\label{thm:char2}
Let $x$ be a sequence taking values in a regular Hausdorff space $X$
such that $\{n: x_n \notin K\} \in \mathcal{I}$ for some compact set $K$.  
Then $\Gamma_x(\mathcal{I})$ is the smallest closed set $C$ such that $\{n: x_n \notin U\} \in \mathcal{I}$ for all sets $U$ containing $C$.
\end{thm}
\begin{proof}
Fix $\kappa \in K$ and define the sequence $y$ by $y_n=\kappa$ if $x_n \notin K$ and $y_n=x_n$ otherwise. 
It follows by Lemma \ref{lem:basic}.\ref{item:6}-\ref{item:5} %, and \ref{item:1} 
that
$
\emptyset \neq \Gamma_x(\mathcal{I})=\Gamma_y(\mathcal{I}) \subseteq K.
$ 
Let also $\mathscr{C}$ be the family of closed sets $C$ such that $\{n: x_n \notin U\} \in \mathcal{I}$ for all open subsets $U\supseteq C$ (note that $\{n: x_n \notin U\} \in \mathcal{I}$ if and only if $\{n: y_n \notin U\} \in \mathcal{I}$).

First, we show that $\Gamma_x(\mathcal{I}) \in \mathscr{C}$. Indeed, $\Gamma_x(\mathcal{I})$ is closed by Lemma \ref{lem:basic}.\ref{item:4}; moreover, let us suppose for the sake of contradiction that there exists an open set $U$ containing $\Gamma_x(\mathcal{I})$ such that $\{n: x_n \notin U\} \in \mathcal{I}^+$, that is, $\{n: y_n \notin U\}=\{n: y_n \in K\setminus U\} \in \mathcal{I}^+$. Considering that $K\setminus U$ is compact, we obtain by Lemma \ref{lem:basic}.\ref{item:6} that there exists an $\mathcal{I}$-cluster point of $y$ in $K\setminus U$. This contradicts the fact that $\Gamma_y(\mathcal{I})=\Gamma_x(\mathcal{I})\subseteq U$.

Lastly, fix $C \in \mathscr{C}$ and let us suppose that $\Gamma_x(\mathcal{I})\setminus C\neq \emptyset$. Fix $\ell \in \Gamma_x(\mathcal{I})\setminus C$ and, by the regularity of $X$, there exist disjoint open sets $U$ and $V$ containing the closed sets $\{\ell\}$ and $K\cap C$, respectively. This is impossible, indeed the set $\{n: x_n \in V\}$ belongs to $\mathcal{I}$ since $C \in \mathscr{C}$, and, on the other hand, it contains $\{n: x_n \in U\} \in \mathcal{I}^+$ since $\ell$ is an $\mathcal{I}$-cluster point.
\end{proof}

\subsection*{Acknowledgments} 
The authors are grateful to Szymon G\l ab (\L{}\'{o}d\'{z} University of Technology, PL) and Ond\v{r}ej Kalenda (Charles University, Prague) for several useful comments.

\bibliographystyle{amsplain}
%\bibliography{ideal}

\end{document}